\documentclass[11pt]{article}
\usepackage{amsmath,amssymb,amsthm, bm}

\newtheorem{theorem}{Theorem}[section]
\newtheorem{proposition}[theorem]{Proposition}
\newtheorem{lemma}[theorem]{Lemma}
\newtheorem{corollary}[theorem]{Corollary}
\theoremstyle{definition}
\newtheorem{definition}[theorem]{Definition}
\newtheorem{example}[theorem]{Example}

\theoremstyle{remark}
\newtheorem{remark}[theorem]{Remark}

\newcommand{\N}{\mathbb{N}}

\newcommand{\R}{\mathbb{R}}
\newcommand{\C}{\mathbb{C}}

\newcommand{\aminus}{{a_{-}}}

\newcommand{\alambda}{{a_{\lambda}}}

\newcommand{\alambdaIzero}{{a_{\lambda}^{\oplus I_0}}}
\newcommand{\aoneone}{{a_{1,1}}}
\newcommand{\aione}{{a_{i, 1}}}
\newcommand{\aij}{{a_{i,j}}}

\newcommand{\aji}{{a_j^{(i)}}}
\newcommand{\azero}{a_0}

\newcommand{\Aone}{{A^1}}
\newcommand{\Alambda}{{A_{\lambda}}}
\newcommand{\Aplus}{{A_+}}
\newcommand{\Aplusone}{{A_{+}^1}}
\newcommand{\Atilde}{{A^{\sim}}}

\newcommand{\Asecond}{{A^{**}}}
\newcommand{\Azero}{{A_0}}
\newcommand{\Azerob}{{A_{0, b}}}
\newcommand{\Azerolambda}{{A_{0,\lambda}}}
\newcommand{\Azerotilde}{{\tilde{A_0}}}

\newcommand{\alphaone}{{\alpha_1}}
\newcommand{\alphatwo}{{\alpha_2}}

\newcommand{\bione}{{b_{i, 1}}}

\newcommand{\bi}{{b_i}}

\newcommand{\bmm}{{b_m}}
\newcommand{\bmp}{{b_m'}}
\newcommand{\bmu}{{b_{\mu}}}
\newcommand{\bminus}{{b_{-}}}

\newcommand{\bzero}{{b_0}}

\newcommand{\Bdelta}{B_{\delta}}
\newcommand{\Bplus}{{B_+}}
\newcommand{\Btilde}{{B^{\sim}}}
\newcommand{\Bdual}{{B^*}}
\newcommand{\Bsecond}{{B^{**}}}
\newcommand{\BMtwo}{{B\otimes M_2}}
\newcommand{\BHA}{{B(\HA)}}
\newcommand{\BcH}{{B(\cH)}}
\newcommand{\Czerooneplus}{{C([0,1])_+}}
\newcommand{\Czerooneplusone}{{C([0,1])_+^1}}
\newcommand{\czeroLambdaA}{{c_0(\Lambda, A)}}
\newcommand{\czeroLambdaAA}{{c_0(\LambdaA, A)}}
\newcommand{\CphiA}{{\mathrm{C}^*(\varphi(A))}}

\newcommand{\limlambda}{\displaystyle \lim_{\lambda}}
\newcommand{\limlambdainfty}{\displaystyle \lim_{\lambda}}
\newcommand{\limlambdaomega}{\displaystyle \lim_{\lambda\to\omega}}
\newcommand{\limminfty}{\displaystyle \lim_{m\to\infty}}
\newcommand{\limninfty}{\displaystyle \lim_{n\to\infty}}
\newcommand{\limmuinfty}{\displaystyle \lim_{\mu}}
\newcommand{\limnuinfty}{\displaystyle \lim_{\nu}}
\newcommand{\limsuplambda}{\displaystyle \limsup_{\lambda}}
\newcommand{\limsupninfty}{\displaystyle\limsup_{n\to\infty}}
\newcommand{\limsupm}{\displaystyle \limsup_m}
\newcommand{\limsupn}{\displaystyle\limsup_n}
\newcommand{\limsupmu}{\displaystyle\limsup_{\mu}}

\newcommand{\llVert}{\left\lVert}
\newcommand{\rrVert}{\right\rVert}

\newcommand{\cH}{\cal H}
\newcommand{\deltaone}{{\delta_1}}

\newcommand{\ei}{{\bm{e}_i}}

\newcommand{\eoneonem}{{e_{1,1}^{(m)}}}
\newcommand{\eonein}{{e_{1,i}^{(n)}}}
\newcommand{\eonejn}{{e_{1,j}^{(n)}}}
\newcommand{\eionen}{{e_{i,1}^{(n)}}}

\newcommand{\eiin}{{e_{i,i}^{(n)}}}
\newcommand{\eijn}{{e_{i,j}^{(n)}}}

\newcommand{\etilde}{{\widetilde{e}}}
\newcommand{\fn}{{f_n}}

\newcommand{\falpha}{{f_{\alpha}}}
\newcommand{\falphaone}{{f_{\alphaone}}}
\newcommand{\falphatwo}{{f_{\alphatwo}}}

\newcommand{\Ftilde}{{F^{\sim}}}
\newcommand{\Fi}{{F_i}}
\newcommand{\Flambda}{{F_{\lambda}}}
\newcommand{\Fmu}{{F_{\mu}}}
\newcommand{\Filambda}{{F_{i, \lambda}}}
\newcommand{\Fimu}{{F_{i, \mu}}}
\newcommand{\galpha}{{g_{\alpha}}}
\newcommand{\galphaone}{{g_{\alphaone}}}
\newcommand{\galphatwo}{{g_{\alphatwo}}}
\newcommand{\HA}{{\mathcal{H}_A}}
\newcommand{\hlambda}{{h_{\lambda}}}

\newcommand{\hlambdamAzero}{{h_{\lambda_{m, A_0}}}}
\newcommand{\hlambdanAzero}{{h_{\lambda_{n, A_0}}}}
\newcommand{\hn}{{h_n}}
\newcommand{\hmu}{{h_{\mu}}}
\newcommand{\hnu}{{h_{\nu}}}
\newcommand{\hvarphi}{{h_{\varphi}}}

\newcommand{\Izero}{{I_{0}}}
\newcommand{\idtwo}{{\id}_{M_2}}
\newcommand{\idzeroone}{{\id}_{[0,1]}}
\newcommand{\idMn}{{\id}_{M_n}}
\newcommand{\Js}{\mathcal{Z}}

\newcommand{\Jzero}{J_{0}}
\newcommand{\kl}{{k_l}}
\newcommand{\km}{{k_m}}

\newcommand{\knp}{{k_n'}}
\newcommand{\knm}{{k_{n_m}}}
\newcommand{\knmp}{{k_{n_m}'}}
\newcommand{\kmu}{{k_{\mu}}}
\newcommand{\kmuzero}{{k_{\mu_0}}}
\newcommand{\kzero}{{k_0}}
\newcommand{\lambdazero}{{\lambda_0}}
\newcommand{\lambdax}{{\lambda_x}}
\newcommand{\lambdatilde}{{\widetilde{\lambda}}}
\newcommand{\lambdanAzero}{{\lambda_{n, A_0}}}
\newcommand{\LambdaA}{{\Lambda_A}}
\newcommand{\LambdaAtilde}{{\Lambda_{\Atilde}}}
\newcommand{\LambdaMN}{{\Lambda_{\MN}}}

\newcommand{\ltwoN}{{\ell^{\rm 2}(\N)}}
\newcommand{\llinfty}{{\ell^{\infty}}}
\newcommand{\linftyLambdaA}{{\ell^{\infty}(\Lambda, A)}}
\newcommand{\linftyLambdaAA}{{\ell^{\infty}(\LambdaA, A)}}
\newcommand{\muzero}{{\mu_0}}
\newcommand{\MCphiA}{{\mathcal{M}(\CphiA)}}
\newcommand{\Mm}{{M_m}}
\newcommand{\Mn}{{M_n}}
\newcommand{\MN}{{M_N}}
\newcommand{\MNlambda}{{M_{N_{\lambda}}}}

\newcommand{\nlambda}{{n_{\lambda}}}

\newcommand{\Cs}{$\mathrm{C}^*$-al\-ge\-bra}
\newcommand{\Css}{$\mathrm{C}^*$-sub\-al\-ge\-bra}

\newcommand{\cC}{\mathcal{C}}
\newcommand{\cCclosure}{\overline{\mathcal{C}}^{\|\cdot\|}}

\newcommand{\cS}{\mathcal{S}}

\newcommand{\cM}{{\mathcal{M}}}
\newcommand{\cMone}{{\mathcal{M}^1}}

\newcommand{\dist}{{\mathrm{dist}}}

\newcommand{\ODphi}{{\mathrm{OD}(\varphi)}}
\newcommand{\ODphiplus}{{\mathrm{OD}(\varphi)_+}}
\newcommand{\ODphiplusone}{{\mathrm{OD}(\varphi)_+^1}}

\newcommand{\ODphik}{{\mathrm{OD}(\varphi, k)}}
\newcommand{\oneA}{{1_{A}}}
\newcommand{\oneAtilde}{{1_{\Atilde}}}
\newcommand{\oneAeone}{{1_{A}\otimes \bm{e}_{1}}}
\newcommand{\oneAei}{{1_{A}\otimes \bm{e}_i}}
\newcommand{\oneAej}{{1_{A}\otimes \bm{e}_j}}
\newcommand{\oneAsecond}{{1_{A^{**}}}}
\newcommand{\oneB}{{1_B}}
\newcommand{\oneBsecond}{{1_{B^{**}}}}
\newcommand{\oneBHA}{{1_{B(\HA)}}}
\newcommand{\onecM}{{1_{\cM}}}
\newcommand{\onen}{{1_{\Mn}}}
\newcommand{\onem}{{1_{\Mm}}}
\newcommand{\oneMN}{{1_{\MN}}}
\newcommand{\OplusAlambda}{{\bigoplus_{\lambda}\Alambda}}
\newcommand{\OplusMNlambda}{{\bigoplus_{\lambda}\MNlambda}}
\newcommand{\OplusFi}{{\displaystyle\bigoplus_{i=0}^d \Fi}}
\newcommand{\OplusFilambda}{{\displaystyle\bigoplus_{i=0}^d \Filambda}}
\newcommand{\OplusFimu}{{\displaystyle\bigoplus_{i=0}^d \Fimu}}
\newcommand{\OplusFmu}{{\bigoplus_{\mu}\Fmu}}
\newcommand{\OplusFlambda}{{\bigoplus_{\lambda}\Flambda}}
\newcommand{\Oplusxii}{{\displaystyle\bigoplus_{i=0}^d\xii}}
\newcommand{\psilambda}{{\psi_{\lambda}}}
\newcommand{\psimu}{{\psi_{\mu}}}
\newcommand{\psitilde}{{\widetilde{\psi}}}
\newcommand{\Phin}{{\Phi_n}}
\newcommand{\Philambda}{{\Phi_{\lambda}}}
\newcommand{\Phii}{{\Phi_i}}
\newcommand{\PhiIzero}{{\Phi_{I_0}}}
\newcommand{\ProdAlambda}{{\prod_{\lambda}\Alambda}}
\newcommand{\ProdMNlambda}{{\prod_{\lambda}\MNlambda}}
\newcommand{\ProdFlambda}{{\prod_{\lambda}\Flambda}}
\newcommand{\ProdFmu}{{\prod_{\mu}\Fmu}}
\newcommand{\PAzero}{P_{A_0}}
\newcommand{\sumvarphii}{{\displaystyle \sum_{i=0}^d\varphii}}
\newcommand{\sumvarphiilambda}{{\displaystyle \sum_{i=0}^d\varphiilambda}}
\newcommand{\sumvarphiimu}{{\displaystyle \sum_{i=0}^d\varphiimu}}
\newcommand{\sigmax}{{\sigma_x}}
\newcommand{\suplambda}{{\displaystyle\sup_{\lambda}}}

\newcommand{\trn}{{\mathrm{tr}_n}}
\newcommand{\trm}{{\mathrm{tr}_m}}

\newcommand{\Trn}{{\mathrm{Tr}_n}}

\newcommand{\uoneone}{{u_{1,1}}}
\newcommand{\uione}{{u_{i,1}}}

\newcommand{\Ux}{{U_x}}
\newcommand{\barUx}{{\overline{U}_x}}

\newcommand{\Uxmu}{{U_{x, \mu}}}
\newcommand{\vi}{{v_i}}
\newcommand{\varphiAzero}{{\varphi|_{A_0}}}
\newcommand{\varphii}{{\varphi_i}}
\newcommand{\varphiilambda}{{\varphi_{i, \lambda}}}
\newcommand{\varphiidtwo}{{\varphi\otimes\idtwo}}
\newcommand{\varepsilonone}{{\varepsilon}_1}
\newcommand{\varphilambda}{{\varphi_{\lambda}}}
\newcommand{\varphimu}{{\varphi_{\mu}}}
\newcommand{\varphilambdam}{{\varphi_{\lambda}^{(m)}}}
\newcommand{\varphiimu}{{\varphi_{i, \mu}}}
\newcommand{\varphilambdatilde}{{\varphi_{\lambdatilde}}}

\newcommand{\varphitilde}{{\widetilde{\varphi}}}
\newcommand{\xii}{{x_i}}
\newcommand{\xIzero}{{x^{\oplus I_0}}}
\newcommand{\xmu}{{x_{\mu}}}
\newcommand{\xalambda}{{x_{a, \lambda}}}

\newcommand{\xn}{{x_n}}

\newcommand{\xnlambda}{{x_{n_{\lambda}}}}

\newcommand{\Xdual}{{X^*}}
\newcommand{\Xsecond}{{X^{**}}}
\newcommand{\XIzero}{{X_{I_{0}}}}
\newcommand{\XIzerodual}{{X_{I_{0}}}^*}
\newcommand{\XIzerosecond}{{X_{I_{0}}^{**}}}

\newcommand{\yn}{{y_n}}

\newcommand{\yIzero}{{y^{\oplus I_0}}}
\newcommand{\yx}{{y_x}}

\DeclareMathOperator{\id}{id}

\DeclareMathOperator{\conv}{conv}
\DeclareMathOperator{\Real}{Re}

\begin{document}
\title{2-positive almost order zero maps  and \\ decomposition rank}
\author{
Yasuhiko Sato \\
}
\date{}

\maketitle
\begin{abstract} 
We consider 2-positive almost order zero (disjointness preserving) maps on \Cs{}s.  Generalizing the argument of  M. Choi for multiplicative domains, we provide an internal characterization of almost order zero for 2-positive maps. 
In addition, it is shown that complete positivity can be reduced to 2-positivity in the definition of decomposition rank for unital separable \Cs s.
\end{abstract}

\section{Introduction}\label{Sec1}

In \cite{WZ}, W. Winter and J. Zacharias provided a structure theorem for completely positive order zero maps, which is based on the work of M. Wolff on disjointness preserving linear maps \cite{Wol}.  Recall that a positive linear map $\varphi : A\rightarrow B$ between two \Cs s is said to have {\it order zero} if $\varphi(a)\varphi(b)=0$ for any positive elements $a$, $b\in A$ with $ab=0$. 
Currently, this concept of order zero maps led to geometric dimensions, known as decomposition rank and nuclear dimension \cite{KW, WZ2}, which both play a crucial role in Elliott's classification program for nuclear \Cs s.  The purpose of this paper is to explore the relationship between 2-positivity and order zero maps. 

In the first part of this paper, we show the one variable characterization of 2-positive almost order zero maps.
\begin{theorem}\label{FirstMainTheorem}
For $\varepsilon >0$ there exists $\delta >0$ satisfying the following condition: for two \Cs s $A$ and $B$, an  approximate unit $\hlambda$, $\lambda\in\Lambda$ of $A$, and a 2-positive contraction $\varphi$ from $A$ to $B$, if a positive contraction $a\in A$ satisfies 
\[ \limsuplambda \llVert \varphi(a)^2 - \varphi(a^2)\varphi(\hlambda) \rrVert <\delta, \]
then the weak*-limit $\hvarphi\in\Bsecond$of $\varphi(\hlambda)$, $\lambda\in\Lambda$ and $a\in A$ satisfy
\[\sup_{b\in A,\ \llVert b\rrVert \leq 1}\llVert \varphi(a)\varphi(b)- \hvarphi\varphi(ab)\rrVert< \varepsilon.\]

\noindent Specifically,  a 2-positive map $\varphi$ from a unital \Cs{} $A$ to a \Cs{} $B$  has order zero if  $\varphi(a)^2=\varphi(a^2) \varphi(\oneA)$ for any positive element $a\in A$. 
\end{theorem}

In the second part of the paper, we study the relation between 2-positivity and decomposition rank. 
The notion of decomposition rank (Definition \ref{Decompositionrank}) was introduced by E. Kirchberg and W. Winter in their work \cite{KW}, in which they showed that finiteness of decomposition rank implies quasidiagonality for \Cs s. In \cite{Win1} W. Winter showed that finiteness of decomposition rank (for separable \Cs s, see \cite{FK} for non-separable cases) also implies the absorption of the Jiang-Su algebra which plays a central role in the recent classification theorem of unital separable simple nuclear \Cs s that satisfy UCT and absorb the Jiang-Su algebra \cite{ET}, \cite{EGLN}, \cite{GLN}, \cite{TiWW}. For unital separable simple nuclear monotracial \Cs s, we showed the converse, i.e., quasidiagonality and Jiang-Su absorption imply finiteness of decomposition rank \cite{MS1, MS2}. 
Our second main result characterizes finiteness of decomposition rank by 2-positive maps instead of completely positive maps.
\begin{theorem}\label{SecondMainTheorem}
Let $A$ be a unital separable \Cs{}. Then the decomposition rank of $A$ is at most $d$ if and only if 
for a finite subset $F$ of contractions in $A$ and $\varepsilon >0$, there exist finite dimensional \Cs s $\Fi$, $i=0,1,...,d$, a 2-positive contraction $\psi : A\rightarrow \OplusFi$, and 2-positive order zero contractions $\varphii : \Fi \rightarrow A$, $i=0,1,...,d $ such that $\sumvarphii :  \OplusFi\rightarrow A$ is contractive and 
\[ \llVert\left(\sumvarphii\right) \circ \psi(x) -x\rrVert < \varepsilon,\quad\text{for all }x\in F.\]
Here we simply write $\sumvarphii\left( \Oplusxii \right) =\sumvarphii(\xii)$ for $\xii\in \Fi$. 
\end{theorem}

Before closing this section, let us collect some notations and terminologies. 

\noindent For a subset $S$ in a vector space, we denote by $\conv S$ the convex hull of $S$.

For a \Cs{} $A$, we let $A_{\rm sa}$ and $\Aplus$ denote the set of self-adjoint elements and the cone of positive elements in $A$. For a subset $S\subset A$,  $S^1$ denotes the set of contractions in $S$. If $A$ is a unital \Cs, $1_A$ denotes the unit of $A$. 

For any two elements $a$ and $b$ in a \Cs{} $A$, we let $[a, b]$ denote the commutator $ab-ba\in A$, and by $a\approx_{\varepsilon} b$ for $\varepsilon>0$ we mean that $\llVert a-b\rrVert <\varepsilon$. 

Unless stated otherwise we consider two \Cs{}s $A$ and $B$, and by a ``map'' $\varphi : A\rightarrow B$ we mean a ``linear map'' from $A$ to $B$.
We let $\id_A$ denote the identity map on $A$, i.e., $\id_A(a)=a$ for any $a\in A$. For $n\in\N$, $M_n$ denotes the \Cs{} of complex $n\times n$ matrices.  A  map $\varphi$ from $A$ to $B$ is called {\it positive} if $\varphi(\Aplus) \subset \Bplus$. For a natural number $k$, a  map $\varphi$ is called {\it $k$-positive} if $\varphi\otimes \id_{M_k}: A\otimes M_k \rightarrow B\otimes M_k$ is positive. If a  map $\varphi: A \rightarrow B$ is k-positive for any $k\in \N$, $\varphi$ is called {\it completely positive}. 

For a positive linear map $\varphi : A\rightarrow B$, the {\it multiplicative domain} of $\varphi$ is defined as the space
$\{ a\in A\ : \ \varphi(ab)=\varphi(a)\varphi(b)\ \text{and } \varphi(ba)=\varphi(b)\varphi(a) \ \text{for any } b\in A\}$.

\section{Orthogonality domains for 2-positive maps}\label{Sec2}
\begin{definition}\label{DefOrthogonality}
Let $A$ and $B$ be two \Cs s, and let $\hlambda \in \Aplusone$, $\lambda\in \Lambda$ be an approximate unit. For a bounded linear map $\varphi$ from $A$ to $B$, we define a subspace $\ODphi$ of $A$ by
\begin{align*} 
\ODphi =\{a \in A\ :\ \varphi(a)\varphi(b)&=\limlambdainfty \varphi(\hlambda)\varphi(ab),\\ \varphi(b)\varphi(a)
&=\limlambdainfty \varphi(ba)\varphi(\hlambda)\quad {\rm for\ any\ } b\in A\}.
\end{align*} 
It follows from the definition that $\limlambdainfty \| [\varphi(\hlambda), \varphi(a)]\| =0$ for any $a\in \ODphi$.
\end{definition}
In this section we mainly deal with 2-positive maps for Kadison's inequality in the following form, which makes $\ODphi$ into a \Cs. 
For two (not necessarily unital) \Cs{}s $A$ and $B$ , if a map $\varphi : A\rightarrow B$ is contractive and 2-positive, then the original Kadison's inequality tells us that 
\[ \varphi\otimes \idtwo \left(\begin{bmatrix}
                     0 & a^* \\
                     a & 0
                   \end{bmatrix}\right)^2
\leq \varphi\otimes \idtwo\left(\begin{bmatrix}
                      0& a^* \\
                     a & 0
                   \end{bmatrix}^2\right)\]
for any $a\in A$, see \cite[p.770]{KR}, for example. Then we have $\varphi(a)^*\varphi(a) \leq \varphi(a^*a)$ for any $a\in A$,  \cite[Corollary 2.8]{Choi}. Let us point out that this inequality also works for non-unital \Cs s. By using this, we can see that $\ODphi$ is a \Cs. 
\begin{proposition}\label{PropODphi}
If a map $\varphi : A\rightarrow B$ is 2-positive, then the following statements hold. \begin{enumerate}
\item $\ODphi$ is a \Cs{} which contains the multiplicative domain of $\varphi$.
\item $\ODphi$ is independent of the choice of the approximate unit.
\end{enumerate}
\end{proposition}
\begin{proof}
Since $\ODphi = \mathrm{OD}(\varphi/ |\varphi\|)$, we may assume $\|\varphi\|\leq 1$ in both {\rm (i)} and {\rm (ii)}.

\noindent {\rm (i)} Since $\varphi$ is a bounded self-adjoint map, it is straightforward to check that $\ODphi$ is a self-adjoint Banach space which contains the multiplicative domain of $\varphi$. It remains to show that $\ODphi$ is closed under multiplication. Let $a$, $b$ be contractions in $\ODphi$, $c$ a contraction in $A$, and $\varepsilon \in (0,1)$. Taking a large $k\in\N$ we have $(1-t^{1/k})t < \varepsilon^2/8$ for any $t\in [0, 1]$. Because of Kadison's inequality and $\|\varphi\|\leq 1$, for any $\lambda\in \Lambda$ with $\|\hlambda^{1/2} a a^*\hlambda^{1/2}- aa^*\| < \varepsilon^2/8$  we have 
\begin{align*}
\|( 1-\varphi(\hlambda)^{1/k})\varphi(ab)\varphi(ab)^*(1-\varphi(\hlambda)^{1/k})\|&\leq \| (1-\varphi(\hlambda)^{1/k})\varphi(aa^*)(1-\varphi(\hlambda)^{1/k})\|\\ &< \|(1-\varphi(\hlambda)^{1/k})\varphi(\hlambda)\| + \varepsilon^2/ 8 < \varepsilon^2 /4.
\end{align*}
Since $\varphi(\hlambda)$, $\lambda\in \Lambda$ almost commutes with $\varphi(a)$, it follows that 
\[ \limlambdainfty \|\varphi(\hlambda)\varphi(ab)\varphi(c) - \varphi(\hlambda)^2\varphi(abc)\| =0,\]
which implies 
$\limlambdainfty \| \varphi(\hlambda)^n (\varphi(ab)\varphi(c) -\varphi(\hlambda)\varphi(abc))\| =0$ for any $n\in\N$.  
 Then we have $\limlambdainfty \|\varphi(\hlambda)^{1/k}(\varphi(ab)\varphi(c) - \varphi(\hlambda)\varphi(abc))\|=0$.
Thus, there exists $\lambdazero \in \Lambda$ such that for any $\lambda\geq \lambdazero$,
\[\|\varphi(ab)\varphi(c) -\varphi(\hlambda)\varphi(abc)\| <\varepsilon.\]
Since $\varepsilon >0$ is arbitrary, we have $\varphi(ab)\varphi(c) = \limlambdainfty \varphi(\hlambda)\varphi(abc)$. 
By $\ODphi^*=\ODphi$, we also have $\varphi(c)\varphi(ab) =\limlambdainfty \varphi(cab)\varphi(\hlambda)$ for any $a$, $b\in \ODphi$ and $c\in A$. 

\noindent {\rm (ii)} Let $\kmu \in \Aplusone$, $\mu\in I$ be another approximate unit of $A$ and let $\ODphik$ be the subspace in Definition \ref{DefOrthogonality} determined by $\{\kmu\}_{\mu\in I}$. Since $\ODphi$ and $\ODphik$ are \Cs{}s, it suffices to show $\ODphiplus \subset \ODphik$. 

Let $a \in \ODphiplusone$, and let $\lambdazero\in \Lambda$ and $\muzero \in I$ be such that $\|\varphi((\hlambda -\kmu)a)\| < \varepsilon$ for any $\lambda \geq \lambdazero$ and $\mu\geq \muzero$. Then it follows that 
\[ \|\varphi(\hlambda -\kmu )\varphi(a)\| =\limnuinfty \|\varphi((\hlambda -\kmu)a)\varphi(\hnu) \| <\varepsilon.\] 
By Kadison's inequality, for any $b\in A^1$ we have 
\[\| \varphi(\hlambda -\kmu )\varphi(ab)\|^2 \leq \|\varphi(\hlambda -\kmu)\varphi(a)\varphi(\hlambda -\kmu)\| < 2\varepsilon,\]
for any $\lambda\geq \lambdazero$ and $\mu\geq \muzero$. Then it follows that $\limmuinfty \varphi(\kmu)\varphi(ab)=\limlambdainfty\varphi(\hlambda)\varphi(ab)=\varphi(a)\varphi(b)$. Since $a$ is self-adjoint, we also see that $\limmuinfty \varphi(ba)\varphi(\kmu) = \varphi(b)\varphi(a)$ for any $b\in A$. 
\end{proof}

To prepare for the Schwartz inequality (Proposition \ref{PropSchwartz}) and the next section, we need the following calculation of non-invertible positive elements. This argument is a slight variation of \cite[Lemma 1.4.4]{Ped}.
\begin{lemma}\label{inverse}
Let $A$ be a \Cs. For two positive elements $a$ and $b$ in the second dual $\Asecond$ with $a\leq b$, there exists a unique contraction $x$ in $\Asecond$ such that $b^{1/2}x = a^{1/2}$ and $p(b) x=x$, where $p(b)$ is the support projection of $b$ defined as the strong limit of $(\frac{1}{n}\oneAsecond + b)^{-1}b\in \Asecond$. If furthermore $[a, b]=0$, then there exists a unique contraction $y$ in $\Asecond$ such that $by=a$ and $p(b)y =y$.

We write $b^{-1/2}a^{1/2}=x$ and $b^{-1}a =y$.
\end{lemma}
\begin{proof}
For $n\in\N$, we set $\xn = (\frac{1}{n}\oneAsecond + b)^{-1/2}a^{1/2} \in \Asecond$. Then it follows that $\xn^*\xn = a^{1/2}(\frac{1}{n}\oneAsecond +b)^{-1} a^{1/2} \leq a^{1/2} (\frac{1}{n}\oneAsecond +a)^{-1} a^{1/2} \leq \oneAsecond$ for any $n\in\N$. Since the unit ball of $\Asecond$ is compact in the $\sigma$-weak (ultraweak) topology, there exists a subnet $\xnlambda$, $\lambda\in\Lambda$ of $\{\xn\}_{n\in\N}$ which converges to a contraction $x\in\Asecond$. Thus we have that 
\[ b^{1/2} x=  \sigma \text{-weak-}\limlambdainfty b^{1/2}\left(\frac{1}{\nlambda}\oneAsecond + b\right)^{-1/2} a^{1/2} = p(b)a^{1/2} =a^{1/2}.\]

If $x'\in\Asecond$ satisfies $b^{1/2}x'=a^{1/2}$ and $p(b)x'=x'$, then we have $x-x'= p(b)(x-x')= \text{strong-}\limninfty$$(\frac{1}{n}\oneAsecond + b)^{-1} b (x-x') =0  $.

In the case of $[a, b]=0$, by a similar argument, we can define a positive contraction $y$ in $\Asecond$ as the strong limit of $a^{1/2}(\frac{1}{n}\oneAsecond + b)^{-1} a^{1/2}$, $n\in\N$. This $y$ also satisfies the desired conditions. 
\end{proof}

\begin{corollary}\label{Corinverse}
Let $A$ and $B$ be \Cs s.
\begin{enumerate}
\item Suppose that  $\varphi : A\rightarrow B$ is a 2-positive map and $a$ and $b$  are two elements in $A$. Then there exists a unique element $\varphi(b^*b)^{-1/2}\varphi(b^*a)\in \Bsecond$ satisfying 
\[ \varphi(b^*b)^{1/2}(\varphi(b^*b)^{-1/2}\varphi(b^*a)) = \varphi(b^*a)\]
and $(\oneBsecond - p(\varphi(b^*b)))\varphi(b^*b)^{-1/2}\varphi(b^*a) =0$.
\item Suppose that  $\varphi : A\rightarrow B$ is a positive map, $x$ is a normal element in $A$, and $y$ is a positive element in $A$ satisfying $xx^*\leq \|x\|^2 y$. Then there exists a unique element $\varphi(y)^{-1/2}\varphi(x) \in \Bsecond$ such that 
\[ \varphi(y)^{1/2}(\varphi(y)^{-1/2}\varphi(x)) = \varphi(x)\quad\text{and}\quad (\oneBsecond-p(\varphi(y)))\varphi^{-1/2}(y)\varphi(x) =0.\]
\end{enumerate}
\end{corollary}
\begin{proof}
In both cases we may assume that $\varphi$ is contractive. We may further assume that $a$ and $x$ are contractions in $A$. 

\noindent {\rm (i)} By Kadison's inequality we have $\varphi(a^*b)^* \varphi(a^*b)\leq \varphi(b^*b)$. From Lemma \ref{inverse}, we obtain the contraction $\varphi(b^*b)^{-1/2}|\varphi(a^*b)|\in \Bsecond$. 
By the polar decomposition of $\varphi(b^*a)$ in $\Bsecond$, there exists a contraction $\varphi(b^*b)^{-1/2}\varphi(b^*a)\in \Bsecond$ satisfying the desired conditions. The uniqueness of $\varphi(b^*b)^{-1/2}\varphi(b^*a)\in\Bsecond$ follows from these conditions automatically. 

\noindent {\rm (ii)} Since $x$ is normal, Kadison's inequality implies that 
\[\varphi(x)\varphi(x)^*\leq \varphi(xx^*)\leq\varphi(y),\]
see \cite[p.770]{KR}. By the same argument as in the proof of  {\rm (i)} we obtain a unique element $\varphi(y)^{-1/2}\varphi(x)\in \Bsecond$ satisfying the desired conditions.
\end{proof}

The following Schwartz inequality was given by M. Choi in \cite[Proposition 4.1]{Choi2} for strictly positive maps and invertible elements. Regarding $\varphi(a^*b)\varphi(b^*b)^{-1}\varphi(b^*a)$ as $(\varphi(b^*b)^{-1/2}$ $\varphi(b^*a))^*\varphi(b^*b)^{-1/2}\varphi(b^*a)$ obtained in Corollary \ref{Corinverse}, we extend his result to the case of non-invertible  elements.

\begin{proposition}\label{PropSchwartz}
Let $A$ and $B$ be \Cs{}s. 
\begin{enumerate}
\item Suppose that $\varphi$ is a 2-positive map from $A$ to $B$. Then for any $a$, $b\in A$ it follows that 
\[ \varphi(a^*b)\varphi(b^*b)^{-1}\varphi(b^*a) \leq \varphi(a^*a).\]
\item Suppose that $\varphi$ is a positive map from $A$ to $B$. Then for a self-adjoint element $x\in A$ and a positive element $y\in A$ with $yx=x$, it follows that 
\[\varphi(x)\varphi(y)^{-1} \varphi(x)\leq \varphi(x^2).\]
\end{enumerate}
\end{proposition}
\begin{proof}
\noindent (i) Since the $2\times 2$ matrix $\begin{bmatrix}
                     \varphi(a^*a) & \varphi(a^*b) \\
                     \varphi(b^*a) & \varphi(b^*b)
                   \end{bmatrix}\in B\otimes M_2$ is positive, the matrix 
\begin{align*}
&\begin{bmatrix}
                     \oneBsecond & 0 \\
                     0 & (\frac{1}{n}\oneBsecond + \varphi(b^*b))^{-1/2}
                   \end{bmatrix}
\begin{bmatrix}
                     \varphi(a^*a) & \varphi(a^*b) \\
                     \varphi(b^*a) & \varphi(b^*b)
                    \end{bmatrix}
\begin{bmatrix}
                      \oneBsecond & 0 \\
                     0 & (\frac{1}{n}\oneBsecond + \varphi(b^*b))^{-1/2}
                   \end{bmatrix}\\
&\quad =\begin{bmatrix}
                     \varphi(a^*a) & \varphi(a^*b)(\frac{1}{n}\oneBsecond + \varphi(b^*b))^{-1/2} \\
                     (\frac{1}{n}\oneBsecond + \varphi(b^*b))^{-1/2}\varphi(b^*a) & (\frac{1}{n}\oneBsecond + \varphi(b^*b))^{-1}\varphi(b^*b) 
                   \end{bmatrix} \quad \in B\otimes M_2
\end{align*}
is also positive for any $n\in\N$. From $\|(\frac{1}{n}\oneBsecond + \varphi(b^*b))^{-1/2}\varphi(b^*a)\|\leq \|\varphi\|^{1/2}\|a\|$ for any $n\in\N$, we obtain an accumulation point $X\in \Bsecond$ of $\{(\frac{1}{n}\oneBsecond + \varphi(b^*b))^{-1/2}\varphi(b^*a)\}_{n\in\N}$ in the sense of $\sigma$-weak topology. It is straightforward to see that $\varphi(b^*b)^{1/2} X = \varphi(b^*a)$ and $(\oneBsecond-p(\varphi(b^*b)))X=0$. By Corollary \ref{Corinverse}, we have $X= \varphi(b^*b)^{-1/2}\varphi(b^*a)$. Then it follows that the $2\times 2$ matrix $\begin{bmatrix}
                     \varphi(a^*a) & (\varphi(b^*b)^{-1/2}\varphi(b^*a))^* \\
                     \varphi(b^*b)^{-1/2}\varphi(b^*a) & p(\varphi(b^*b))
                   \end{bmatrix}\in \Bsecond\otimes M_2$ is also positive.
Because of 
\begin{align*}
0\leq&\begin{bmatrix}
                     \oneBsecond & 0 \\
                     -\varphi(b^*b)^{-1/2}\varphi(b^*a) & \oneBsecond
                   \end{bmatrix}^*
\begin{bmatrix}
                     \varphi(a^*a) & (\varphi(b^*b)^{-1/2}\varphi(b^*a))^* \\
                     \varphi(b^*b)^{-1/2}\varphi(b^*a) & p(\varphi(b^*b))
                    \end{bmatrix}\\
&\cdot \begin{bmatrix}
                      \oneBsecond & 0 \\
                     -\varphi(b^*b)^{-1/2}\varphi(b^*a) & \oneBsecond
                   \end{bmatrix}
=\begin{bmatrix}
                     \varphi(a^*a)-\varphi(a^*b)\varphi(b^*b)^{-1}\varphi(b^*a) &0 \\
  0 & p(\varphi(b^*b)) 
                   \end{bmatrix},
\end{align*}
we conclude that 
\[\varphi(a^*a)\geq \varphi(a^*b)\varphi(b^*b)^{-1}\varphi(b^*a).\]

\noindent {\rm (ii)} When $yx=x$, the $2\times 2$ matrix 
$\begin{bmatrix}
                     x^2 & x \\
                     x & y
                   \end{bmatrix}\in A\otimes M_2$ is positive. By \cite[Corollary 4.4]{Choi2}, we can see that $\begin{bmatrix}
                     \varphi(x^2) & \varphi(x) \\
                     \varphi(x) & \varphi(y)
                   \end{bmatrix}\in B\otimes M_2$ is also positive, even for a positive map $\varphi$. By the same argument as the proof of {\rm (i)}, we conclude that $\varphi(x^2)\geq \varphi(x)\varphi(y)^{-1}\varphi(x)$.
\end{proof}

\section{ Proof of Theorem \ref{FirstMainTheorem} and applications}

In the following lemma, for a unital \Cs{} $A$ we denote by $\HA$ the separable Hilbert $A$-module $A\otimes \ltwoN$ and by $\langle\ \cdot \ , \ \cdot\ \rangle_{\HA} : \HA\times\HA\rightarrow A$ the inner product on $\HA$, which is defined by 
\[ \langle \bm{x} , \ \bm{y} \rangle_{\HA}= \sum_{i=1}^{\infty} x_i^*y_i\in A,\quad\text{for }\ \bm{x}=(x_i)_{i\in\N}\  \text{  and  }\ \bm{y}=(y_i)_{i\in\N}\in \HA,\]
(see \cite{Kas}, \cite{Lan} for detail). Let $\BHA$ denote the set of adjointable operators on $\HA$. We let $\{\ei\}_{i \in \N}$ denote the canonical orthonormal basis of $\ltwoN$, and regard $a\in \BHA$ as an $\infty$-matrix whose $(i,j)$-entry is $\aij := \langle \oneAei  , \ a (\oneAej)\rangle_{\HA} \in A$ for $i$, $j\in\N$. 
\begin{lemma}\label{PositiveUnitary}
Let $A$ be a unital \Cs. For $\varepsilon >0$ the following statements hold. 
\begin{enumerate}
\item If a positive contraction $a\in \BHA$ satisfies $\| a_{1,1}\| <\varepsilon$, then $\displaystyle \left\lVert \sum_{i=1}^{\infty} \aione^*\aione \right\rVert <\varepsilon$. 
\item If a unitary $u \in \BHA$ satisfies $\|\uoneone^*\uoneone -\oneA\|<\varepsilon$, then $\displaystyle \left\lVert \sum_{i=2}^{\infty} \uione^*\uione \right\rVert <\varepsilon$.
\end{enumerate}
\end{lemma}
\begin{proof}
\noindent {\rm (i)} For a positive contraction $a\in \BHA$, we have an element $b\in \BHA$ with $b^*b=a$, which implies that $\displaystyle \aoneone =\sum_{i=1}^{\infty}\bione^*\bione$, where the right hand side is in the operator norm topology on $A$.  Then we have that 
\begin{align*}
\llVert \sum_{i=1}^{\infty} \aione^*\aione\rrVert
&= \llVert \langle  \oneAeone  , \ a^*a (\oneAeone) \rangle_{\HA}\rrVert \\
&\leq \llVert b\rrVert ^2\llVert \langle \oneAeone , \ b^*b (\oneAeone) \rangle_{\HA}\rrVert
\leq \llVert \sum_{i=1}^{\infty} \bione^*\bione\rrVert < \varepsilon.
\end{align*}
\noindent {\rm (ii)} From $u^*u=\oneBHA$, it follows that $\displaystyle \sum_{i=1}^{\infty} \uione^*\uione =\oneA$ in the operator norm topology. Then we have that 
\[ \llVert\sum_{i=2}^{\infty} \uione^*\uione \rrVert =\llVert \oneA - \uoneone^*\uoneone\rrVert < \varepsilon. \]
\end{proof}

\begin{remark}\label{RemarkApprox}
In Theorem \ref{FirstMainTheorem}, the existence of the weak*-limit $\hvarphi$ in $\Bsecond$ of $\varphi(\hlambda)$, $\lambda\in \Lambda$ follows from the boundedness and monotonicity of $\varphi(\hlambda)$, $\lambda\in\Lambda$ (see Lemma 2.4.19 \cite{BR} for example). Besides this weak*-limit $\hvarphi$ is independent of the choice of approximate unit of $A$. Actually, taking another approximate unit $\kmu$, $\mu\in I$ of $A$, for any $\lambda\in\Lambda$ and $\varepsilon >0$ we obtain $\muzero\in I$ such that $\hlambda \leq \kmuzero^{1/2}\hlambda\kmuzero^{1/2} + \varepsilon \oneAsecond\leq \kmuzero +\varepsilon \oneAsecond$. Since
$\varphi$ is positive and contractive, it follows that $\varphi(\hlambda)\leq \varphi(\kmu) +\varepsilon \oneBsecond$ for any $\mu\geq \muzero$. Then the weak*-limit of $\varphi(\kmu)$ is larger than $\hvarphi$.
\end{remark}

The following two lemmas show that a given approximate unit in Theorem \ref{FirstMainTheorem} can be reduced to the case with a special property for separable \Cs s. 
\begin{lemma}\label{convex}
Let $X$ be a Banach space and $\Phii : \Xsecond\rightarrow\Xsecond$, $i\in I$ weak*-continuous affine maps such that $\Phii(X)\subset X$ for any $i\in I$, and let $\delta > 0$. Suppose that a net $\alambda$, $\lambda\in\Lambda$ of contractions in $X$ converges to $x\in \Xsecond$ in the weak*-topology and that $\llVert \Phii(x)\rrVert <\delta$ for any $i\in I$. 
Then there exists a net $\bmu$, $\mu\in J$ in $\conv \{\alambda\ : \ \lambda\in\Lambda\}\subset X$ which converges to $x$ in the weak*-topology and satisfies
\[\limsupmu \llVert \Phii(\bmu)\rrVert \leq \delta \quad\text{ for any }\ i\in I.\]
\end{lemma}
\begin{proof}
Let $\Izero$ be a finite subset of $I$, $\Jzero$ a finite subset of $\Xdual$, and $\varepsilon>0$. By the assumption of $\alambda$ we have $\lambdazero\in\Lambda$ such that $|\varphi(x-\alambda)|<\varepsilon$ for any $\lambda\geq\lambdazero$ and $\varphi \in \Jzero$. We set $\XIzero = \bigoplus_{i\in \Izero} X$ and $\yIzero = (y, y, ..., y)\in \XIzerosecond (\cong\bigoplus_{i\in \Izero}\Xsecond)$ for $y\in \Xsecond$. Define a weak*-continuous affine map $\PhiIzero : \XIzerosecond\rightarrow \XIzerosecond$ by $\PhiIzero ((y_i)_{i\in\Izero})=(\Phii(y_i))_{i\in\Izero}$ for $(y_i)_{i\in\Izero}\in \XIzerosecond$, and 
\[\cC =\conv \ \{\PhiIzero(\alambdaIzero) \ :\ \lambda\geq \lambdazero\}\subset \XIzero.\] 
We let $\Bdelta$ be the open ball of radius $\delta>0$ in $\XIzero$ and we denote by $\cCclosure$ the norm closure of $\cC$ in $\XIzero$.

If we assume that $\cCclosure \cap \Bdelta =\emptyset$,
then by the Hahn-Banach theorem we obtain $\psi \in \XIzerodual$ and $t\in\R$ such that 
$\|\psi\|=1$ and 
\[\Real \psi(c) \geq t > \Real \psi (b)\quad\text{for any } \ c\in\cCclosure \text{ and } b\in \Bdelta.\]
Then it follows that 
\[|\psi(c)| \geq \Real \psi(c) \geq \delta\quad\text{ for any }c\in \cCclosure.\]
However the net $\PhiIzero (\alambdaIzero)\in\cC$, $\lambda\geq\lambdazero$ converges to $\PhiIzero (\xIzero)\in\XIzerosecond$ in the weak*-topology on $\XIzerosecond$, which implies that 
\[ \delta > \llVert\PhiIzero (\xIzero)\rrVert \geq \limlambda |\psi(\PhiIzero(\alambdaIzero))| \geq \delta,\]
a contradiction. Hence we have $\cCclosure \cap \Bdelta \neq \emptyset$.

We define the ordered set $J$ by 
\[J =\{\Izero \subset I\ :\ |\Izero|< \infty\}\times\{\Jzero\subset \Xdual\ :\ |\Jzero|<\infty\}\times\{\varepsilon\in\R\ : \ \varepsilon>0\},\]
with the inclusion orders on finite subsets and the reverse order on $\R$.
For $\mu=(\Izero, \Jzero, \varepsilon)\in J$, now we obtain $\bmu \in \conv\{\alambda\ :\ \lambda\geq\lambdazero
\}\subset X$ such that $\llVert\Phii(\bmu)\rrVert < \delta$ for all $i\in\Izero$ and $|\varphi(x-\bmu)|<\varepsilon$ for all $\varphi\in \Jzero$. This net $\bmu\in X$, $\mu\in J$ satisfies the required condition.
\end{proof}

\begin{lemma}\label{LemSep}
Let $A$ and $B$ be two separable \Cs s, $\varphi$ a 2-positive contraction from $A$ to $B$, and $\delta >0$. Suppose that a positive contraction $a\in A$ and an approximate unit $\hn$, $n\in\N$ of $A$ satisfy 
\[\limsupn \llVert\varphi(a)^2-\varphi(a^2)\varphi(\hn)\rrVert <\delta. \]
Then there exists an approximate unit $\km$, $m\in\N$ of $A$ such that for any $b\in \Aplusone$ and $\varepsilon >0$ there exists $\bminus\in\Aplusone$ satisfying $\llVert b- \bminus\rrVert<\varepsilon$ and
$\km\bminus =\bminus$ for a large $m\in\N$ and that $\limsupm\llVert\varphi(a)^2-\varphi(a^2)\varphi(\km)\rrVert <\delta$.
\end{lemma}

\begin{proof}
Since $A$ is separable, there exists a strictly positive element $\kzero \in A$ (see \cite[Section 3.10]{Ped} for the definition). Set $\fn(t)=\min \{ \max \{0, nt-1\}, 1\}$, $t\in\R$ and $\knp =\fn(\kzero)\in A$, $n\in\N$. For a norm dense subset $\{\bmp\}_{m\in\N}$ in $\Aplusone$, we see that $\{\knp\bmp \knp\in \Aplusone \ : \ m, n\in\N\}$ is also norm dense in $\Aplusone$. By reindexing $\{\knp\bmp \knp \}_{m, n\in\N}$ and taking a subsequence $n_m\in\N$, $m\in\N$, we obtain a norm dense subset $\{\bmm\}_{m\in\N}$ in $\Aplusone$ such that $\knmp\bi =\bi$ for all $i=1, 2, ..., m$.

Let $\delta'>0$ be such that $\limsupn \llVert\varphi(a)^2-\varphi(a^2)\varphi(h_n)\rrVert <\delta'<\delta$. 
Since $\varphi(a)^2-\varphi(a^2)\varphi(\hn)$ converges to $\varphi(a)^2-\varphi(a^2)\hvarphi$ in the weak*-topology in $\Bsecond$, it follows that $\llVert\varphi(a)^2-\varphi(a^2)\hvarphi\rrVert \leq\delta'$. 
By Remark \ref{RemarkApprox}, we see that $\varphi(\knm')$, $m\in\N$ converges to $\hvarphi$ in the weak*-topology. We define a weak*-continuous affine map $\Phi : \Bsecond\rightarrow \Bsecond$ by $\Phi(x)=\varphi(a)^2-\varphi(a^2)x$ for $x\in\Bsecond$. For $l\in\N$, applying Lemma \ref{convex}  to $\varphi(\knmp)\in B$, $m\geq l$ inductively, we can obtain $\kl\in \conv \{\knmp\ :\ m\geq l\}\subset A$ such that 
\[\llVert\Phi(\varphi(\kl))\rrVert\leq \delta' + \frac{1}{l} <\delta\quad \text{and}\quad \kl k_{l-1}=k_{l-1},\]
 where we set $k_0=0$.  Since $\kl \bi=\bi$ for $i=1,2,...,l$ and $\{\bmm\}_{m\in\N}$ is norm dense in $\Aplusone$, the sequence $\kl\in A$, $l\in\N$ satisfies the required condition.
\end{proof}

\begin{proposition}\label{separable}
Theorem \ref{FirstMainTheorem} holds for the case that $A$ is separable.
\end{proposition}
\begin{proof}
For $\alpha\in (0,1)$ and $t\in [0,1]$, we set $\falpha (t) =\min\{\max\{0,\ \alpha^{-1} t-1\},\ 1\}$. Since  $\falpha|_{[0, \alpha]}=0$, there exists $\galpha\in \Czerooneplus$ such that $\galpha\cdot\idzeroone =\falpha$. Here $\idzeroone\in\Czerooneplusone$ means the continuous function defined by $\idzeroone(t)=t$ for $t\in [0,1]$. 

For $\varepsilon \in(0, 1)$ we let $\alphaone \in (0,1)$ be such that $\llVert\idzeroone\cdot (1-\falphaone)\rrVert < \varepsilon^2/16$. Set $\varepsilonone = (\varepsilon/(8\|\galphaone\|))^2 >0$. Let $\alphatwo \in (0, 1/4)$ be such that $\llVert \idzeroone\cdot (1-\falphatwo)\rrVert < \varepsilonone /4$, and let $\deltaone>0$ be such that 
$\deltaone < \varepsilonone /(4\|\galphatwo\|)$. By approximating $(\idzeroone)^{1/2}$ with polynomials, we let $\delta\in (0, \deltaone)$ be such that for any positive contractions $x$, $y$ in a \Cs{}, the condition $\|[x, y]\| <6\delta$ implies $\|[x^{1/2}, y]\| < \deltaone$. 

Let $A$ be a separable \Cs, $B$ a \Cs , $\varphi : A\rightarrow B$ a 2-positive contraction, and $\hn\in \Aplusone$, $n\in\N$ an approximate unit of $A$. Suppose that a positive contraction $a\in A$ satisfies $\limsupn \|\varphi(a)^2-\varphi(a^2)\varphi(\hn) \|< \delta$. Note that from $\limsupn \|[\varphi(a^2), \varphi(\hn)]\|< 2\delta$, it follows that $\limsupn \|[\varphi(a)^2, \varphi(h_n)]\|<6\delta$, and then $\limsupn \| [\varphi(a), \varphi(\hn)]\| \leq\deltaone$. By Lemma \ref{LemSep}, we obtain an approximate unit $\km$, $m\in\N$ of $A$ such that for any $b\in \Aplusone$ and $r\in(0, \varepsilon)$ there exists $\bminus\in\Aplusone$ with $\llVert b-\bminus\rrVert < r$ and $\km\bminus =\bminus$ for a large $m\in\N$, and that $\limsupm\llVert\varphi(a)^2-\varphi(a^2)\varphi(\km)\rrVert <\delta$. For a small $r>0$, it follows that $\limsupm\llVert\varphi(\aminus)^2-\varphi(\aminus^2)\varphi(\km)\rrVert <\delta$, and $\km\aminus =\aminus$ for a large $m\in\N$.  Note that $\hvarphi$ coincides with the weak*-limit of $\varphi(\km)$ from Remark \ref{RemarkApprox}. 
Once we show \[\sup_{b\in \Aplusone} \llVert\varphi(\aminus)\varphi(\bminus)-\hvarphi\varphi(\aminus\bminus)\rrVert <\varepsilon,\]
it follows that 
\[\sup_{b\in\Aplusone} \llVert\varphi(a)\varphi(b)-\hvarphi\varphi(ab)\rrVert < 5\varepsilon. \]
Thus we may reduce to the case that $a$, $b\in\Aplusone$ satisfy $\hn a=a$ and $\hn b=b$ for a large $n\in\N$.

Let $b$ be a positive contraction in $A$ with $ \hn b= b$ for a large $n\in \N$. Set self-adjoint elements 
\[x=\begin{bmatrix}
                     0 & a \\
                     a & b
                   \end{bmatrix},\quad 
\yn=\begin{bmatrix}
                     \hn & 0 \\
                     0 & \hn
                   \end{bmatrix}\in A\otimes M_2.\]
Now we have $\yn x=x$ for a large $n\in\N$, then {\rm (ii)} of Proposition \ref{PropSchwartz} implies that 
\[ \varphiidtwo (x) \varphiidtwo (\yn)^{-1} \varphiidtwo(x) \leq \varphiidtwo (x^2). \]
This inequality implies that 
$\begin{bmatrix}
                     \varphi(a^2) & \varphi(ab) \\
                     \varphi(ba) &  \varphi(a^2+b^2)
                   \end{bmatrix}$
\begin{align*}
&\geq\left(
\begin{bmatrix}
                     \varphi(\hn) & 0 \\
                     0 & \varphi(\hn)
                   \end{bmatrix}^{-1/2}
\begin{bmatrix}
                     0 & \varphi(a) \\
                     \varphi(a) & \varphi(b)
                   \end{bmatrix}\right)^*
\begin{bmatrix}
                     \varphi(\hn) & 0 \\
                     0 & \varphi(\hn)
                   \end{bmatrix}^{-1/2}
\begin{bmatrix}
                     0 & \varphi(a) \\
                     \varphi(a) & \varphi(b)
                   \end{bmatrix} \\
&\geq
\begin{bmatrix}
                     0 & \varphi(a) \\
                     \varphi(a) & \varphi(b)
                   \end{bmatrix}
\begin{bmatrix}
                     \galphatwo(\varphi(\hn)) & 0 \\
                     0 & \galphatwo(\varphi(\hn))
                   \end{bmatrix}
\begin{bmatrix}
                     0 & \varphi(a) \\
                     \varphi(a) & \varphi(b)
                   \end{bmatrix}\\
&=
\begin{bmatrix}
                     \varphi(a)\galphatwo(\varphi(\hn))\varphi(a) & \varphi(a)\galphatwo(\varphi(\hn))\varphi(b) \\
                     \varphi(b)\galphatwo(\varphi(\hn))\varphi(a) &\quad \varphi(a)\galphatwo(\varphi(\hn))\varphi(a) + \varphi(b)\galphatwo(\varphi(\hn))\varphi(b)
                   \end{bmatrix}.
\end{align*}
Set $y=\varphi(a^2+b^2)-\varphi(a)\galphatwo(\varphi(\hn))\varphi(a) + \varphi(b)\galphatwo(\varphi(\hn))\varphi(b)$. 
Then the following matrix $X\in \BMtwo$ is a positive element, 
\[X=\begin{bmatrix}
                     \varphi(\hn)(\varphi(a^2)-  \varphi(a)\galphatwo(\varphi(\hn))\varphi(a))\varphi(\hn) & \varphi(\hn)(\varphi(ab)- \varphi(a)\galphatwo(\varphi(\hn))\varphi(b))\varphi(\hn) \\
                     \varphi(\hn)(\varphi(ba)-\varphi(b)\galphatwo(\varphi(\hn))\varphi(a))\varphi(\hn) &  \varphi(\hn)\cdot y\cdot\varphi(\hn)
                   \end{bmatrix}.\]
By the choice of $\galphatwo$, $\falphatwo$, $\deltaone >\delta$,  and  $\alphatwo\in(0, 1/4)$ we have that
\begin{align*}
&\limsupn\llVert \varphi(\hn )(\varphi(a^2)-\varphi(a)\galphatwo(\varphi(\hn))\varphi(a))\varphi(\hn)\rrVert \\
&\leq \limsupn\llVert \varphi(\hn)\varphi(a^2)\varphi(\hn) -\varphi(a)\falphatwo(\varphi(\hn))\varphi(\hn)\varphi(a)\rrVert + 2\deltaone\|\galphatwo\|\\
&\leq \limsupn\llVert\varphi(\hn)\varphi(a^2)\varphi(\hn) - \varphi(a)\varphi(\hn)\varphi(a)\rrVert + \frac{\varepsilonone}{4}+ \frac{\varepsilonone}{2}\\
&\leq \limsupn\llVert\varphi(a^2)\varphi(\hn) - \varphi(a)^2\rrVert + \deltaone + \frac{\varepsilonone}{4}+ \frac{\varepsilonone}{2}\\
&< 2\deltaone + \frac{\varepsilonone}{4} + \frac{\varepsilonone}{2} < \varepsilonone.
\end{align*}
Applying {\rm (i)} of Lemma \ref{PositiveUnitary} to $X/\|X\|\in \Btilde\otimes M_2$, we have that 
\[\limsupn \llVert \varphi(\hn)(\varphi(ba) -\varphi(b)\galphatwo(\varphi(\hn))\varphi(a))\varphi(\hn)\rrVert^2 < \|X\|\varepsilonone < 5\varepsilonone.\] 
Then it follows that 
\begin{align*}
&\limsupn \llVert \varphi(\hn)(\varphi(ba)\varphi(\hn) - \varphi(b)\varphi(a))\rrVert < \sqrt{\varepsilonone}\left(\sqrt{5} + \frac{1}{2}\right) + \frac{\varepsilonone}{4}, \\ 
&\limsupn \llVert \falphaone (\varphi(\hn))(\varphi(ba)\varphi(\hn) -\varphi(b)\varphi(a))\rrVert < \|\galphaone\|\left(\sqrt{\varepsilonone}\left(\sqrt{5} + \frac{1}{2}\right) + \frac{\varepsilonone}{4}\right) < \frac{\varepsilon}{2}.
\end{align*}
By Kadison's inequality and $b^2\leq \hn$ for a large $n\in \N$, we see that 
\begin{align*}
&\llVert\falphaone(\varphi(\hn))\varphi(ba)-\varphi(ba)\rrVert^2\leq \llVert (1-\falphaone(\varphi(\hn))\varphi(\hn)\rrVert < \frac{\varepsilon^2}{16},\\
&\llVert\falphaone(\varphi(\hn))\varphi(b)\varphi(a)-\varphi(b)\varphi(a)\rrVert^2\leq \llVert (1-\falphaone(\varphi(\hn))\varphi(\hn)\rrVert < \frac{\varepsilon^2}{16}.
\end{align*}
Therefore we conclude that 
\[ \limsupn\llVert\varphi(ba)\varphi(\hn)-\varphi(b)\varphi(a)\rrVert < \varepsilon, \]
which implies
\[ \llVert\varphi(ba)\hvarphi-\varphi(b)\varphi(a)\rrVert < \varepsilon. \]
\end{proof}

\begin{proof}[Proof of Theorem \ref{FirstMainTheorem}]
For $\varepsilon >0$ we obtain $\delta >0$ in Proposition \ref{separable} satisfying the condition for separable \Cs s. Let $A$, $B$ be (not-necessarily separable) \Cs s, $\varphi : A\rightarrow B$ a 2-positive contraction, and $\hlambda\in\Aplusone$, $\lambda\in\Lambda$ an approximate unit as in the theorem. Suppose that a positive contraction $a\in A$ satisfies 
\[\limsuplambda \llVert \varphi(a)^2- \varphi(a^2)\varphi(\hlambda) \rrVert < \delta.\]

Let $\cS$ be the set of all separable \Css s $\Azero$ of $A$ such that $\Azero\ni a$. For $\Azero\in \cS$ we have an increasing sequence $\lambdanAzero \in \Lambda$, $n\in\N$ such that $\limninfty \llVert \azero -\hlambdanAzero \azero\rrVert=0$ for any $\azero\in \Azero$,  $\limsupninfty \llVert\varphi(a)^2-\varphi(a^2)\varphi(\hlambdanAzero)\rrVert <\delta$, and $\llVert\hlambdanAzero \hlambdamAzero-\hlambdamAzero\rrVert< 1/n$ for all $m=1,2,...,n-1$. We  denote by $\Azerotilde\in\cS$ the \Css\ of $A$ generated by $\Azero$ and $\{\hlambdanAzero\}_{n\in\N}$. Note that $\hlambdanAzero$, $n\in\N$ is an approximate unit of $\Azerotilde$.  We let $\PAzero$ be the weak*-limit of $\varphi(\hlambdanAzero)$, $n\in\N$ in $\Bsecond$. By the condition of $\delta >0$ in Proposition \ref{separable}, it follows that 
\[\sup_{b\in \Azerotilde, \ \llVert b\rrVert \leq 1}\llVert \varphi(a)\varphi(b) -\PAzero\varphi(ab)\rrVert < \varepsilon.\]

Regarding $\PAzero \in\Bsecond$, $\Azero\in\cS$ as a net by the inclusion order of $\cS$, we can see that $\PAzero$, $\Azero\in\cS$ converges to $\hvarphi$ in the weak*-topology of $\Bsecond$. Actually, for any $\lambda\in\Lambda$ there exists $\Azerolambda\in\cS$ such that $\hlambda\in\Azerolambda$, then it follows that 
$\varphi(\hlambda)\leq \PAzero \leq\hvarphi$ for any $\Azero \in \cS$ with $\Azero\supset \Azerolambda$. This implies that $|\psi(\hvarphi -\PAzero)|\to 0$ for any $\psi\in \Bdual$. 

For any $b\in \Aone$ we let $\Azerob\in\cS$ be such that $b\in\Azerob$. For $\Azero\in\cS$ with $\Azero\supset \Azerob$ we have seen that $\llVert\varphi(a)\varphi(b) -\PAzero \varphi(ab)\rrVert<\varepsilon$. Since $\PAzero\varphi(ab)$ converges to $\hvarphi\varphi(ab)$ in the weak*-topology of $\Bsecond$, it follows that 
\[\llVert\varphi(a)\varphi(b) -\hvarphi\varphi(ab)\rrVert\leq\varepsilon\quad\text{for any } b\in \Aone.\] 
\end{proof}

Theorem \ref{FirstMainTheorem} can be used to give an alternative proof of the structure theorem for completely positive order zero maps  \cite{Wol}, \cite{WZ}, \cite{Fara}. Our approach is effective even for 2-positive maps. 
\begin{corollary}\label{CorStructure}
Let $A$, $B$ be two \Cs s, and $\hlambda\in \Aplusone$, $\lambda\in\Lambda$ be an approximate unit of $A$. Suppose that $\varphi$ is a 2-positive map from $A$ to $B$ such that 
\[ \varphi(a)^2=\limlambdainfty\varphi(a^2)\varphi(\hlambda) \quad (\text{in the operator norm topology}), \]
for any positive element $a\in A$. Then there exist a $*$-homomorphism $\pi$ from $A$ to $\Bsecond$ and a positive element $\hvarphi$ in the multiplier algebra $\MCphiA$ of $\mathrm{C}^*(\varphi(A))$ such that 
\[ \pi(a) \in \MCphiA\cap\{\hvarphi\}'\quad \text{ and}\quad \varphi(a)=\hvarphi\pi(a),\]
for any $a\in A$. In particular, $\varphi$ is completely positive.
\end{corollary}
\begin{proof}
We may assume that $\varphi$ is contractive.

We set $\hvarphi$ be the weak*-limit in Remark \ref{RemarkApprox}. Since $\hvarphi\varphi(a)=\varphi(a^{1/2})^2=\varphi(a)\hvarphi$ for any $a\in\Aplusone$, it follows that $\hvarphi\in\MCphiA\cap (\CphiA)'$. By Lemma \ref{inverse} and by $\hvarphi\geq \varphi(a)$ for any $a\in \Aplusone$, we can define a positive element $\pi(a) = \hvarphi^{-1}\varphi(a) \in \Bsecond$ for any $a\in \Aplusone$. Set $f_n(\hvarphi)=\left(\frac{1}{n}\oneBsecond + \hvarphi\right)^{-1}\in \MCphiA\subset \Bsecond$ for $n\in\N$. Note that  for $a$, $b\in \Aplusone$ and $m$, $n\in\N$ 
\[ \llVert \varphi(a)\left(f_n(\hvarphi) - f_m(\hvarphi)\right)\varphi(b) \rrVert ^2 \leq \llVert \hvarphi^4 \left(f_n(\hvarphi) - f_m(\hvarphi)\right)^2 \rrVert,\]
then, by Dini's theorem, $\varphi(a)f_n(\hvarphi)\varphi(b)\in \CphiA$ converges to $\varphi(a)\hvarphi^{-1}\varphi(b)$ in the operator norm topology. Thus we have $\hvarphi^{-1}\varphi(a)\in \MCphiA$ for any $a\in \Aplusone$. 

By the uniqueness of $\hvarphi^{-1}\varphi\left(\frac{a+b}{\|a\|+\|b\|}\right)$ for $a$, $b\in \Aplus$ in Lemma \ref{inverse}, it follows that $\pi\left(\frac{a+b}{\|a\|+\|b\|}\right) =\pi\left(\frac{a}{\|a\|+\|b\|}\right)+\pi\left(\frac{b}{\|a\|+\|b\|}\right)$. Considering the linear span of $\Aplusone$, we obtain a self-adjoint linear map $\pi : A\rightarrow \MCphiA$. Applying Theorem \ref{FirstMainTheorem} to $\varphi $, for $a$, $b\in\Aplus$ we have $\varphi(a)\varphi(b)=\hvarphi\varphi(ab)$, which implies that $\pi(a)\pi(b) =\pi(ab)$.
\end{proof}

The following result has a similar flavor to the fact that a 2-quasitrace is an n-quasitrace for \Cs s \cite{BH}, it would be interesting to know what the natural relation is.

\begin{corollary}\label{Cor2positiveorderzero}
Every 2-positive order zero map is completely positive. 

\noindent More generally, a 2-positive map is completely positive if its  restriction to any commutative \Css{} is order zero.
\end{corollary}
\begin{proof}
Let $\varphi : A\rightarrow B$ be a 2-positive map between two \Cs s. 
If $\varphiAzero$ is completely positive for any separable \Css\ $\Azero$ of $A$, then $\varphi$ itself is completely positive.
 Thus we may assume that $A$ is separable. 
 
 By Corollary \ref{CorStructure}, it suffices to show that an approximate unit $\hn$ $n\in\N$ of $A$ satisfies $\varphi(a)^2=\limninfty \varphi(a^2)\varphi(\hn)$ for any $a\in\Aplusone$. By the same argument as in the proof of Lemma \ref{LemSep}, we can find an approximate unit $\hn$, $n\in\N$ of $A$ such that for any $a\in\Aplusone$ and $\varepsilon >0$ there exist $\aminus\in\Aplusone$ and $N\in\N$ satisfying $\llVert a-\aminus\rrVert<\varepsilon$ and $\hn \aminus=\aminus$ for $n\geq N$. For $n\geq N$, set $\cC$ be the commutative \Css\ of $A$ generated by $\aminus$ and $\hn$. 
 By the assumption $\varphi|_\cC$ is an order zero completely positive map. Thus it follows that $\varphi(\aminus)^2=\varphi(\aminus^2)\varphi(\hn)$, which implies that $\limsupn \llVert\varphi(a)^2-\varphi(a^2)\varphi(\hn)\rrVert\leq 4\varepsilon\llVert\varphi\rrVert^2$. Since $\varepsilon>0$ is arbitrary, we  conclude that $\varphi(a)^2=\limninfty\varphi(a^2)\varphi(\hn)$ for any $a\in \Aplusone$.
\end{proof}

Combining the proof above with Corollary \ref{CorStructure}, we  see the following structure theorem.
\begin{corollary}
Let $A$ and $B$ be two \Cs s. For a 2-positive order zero map $\varphi : A\rightarrow B$, there exist a representation $\pi$ of $A$ on $\Bsecond$, and a positive contraction $\hvarphi\in\Bsecond$ satisfying the same condition in Corollary \ref{CorStructure}.
\end{corollary}

The next result is motivated by the question in \cite[Section 5]{JK} for general \Cs s.
\begin{corollary}
Let $A$ and $B$ be \Cs s, and let $\hlambda\in\Aplusone$, $\lambda\in\Lambda$ be an approximate unit of $A$. For a 2-positive linear map $\varphi$ from $A$ to $B$, the following holds.
\[\ODphi= \mathrm{span}\{a\in \Aplusone\ :\ \varphi(a)^2 =\limlambdainfty \varphi(a^2)\varphi(\hlambda)\}.\]
\end{corollary}
\begin{proof}
From Theorem \ref{FirstMainTheorem}, the right hand side is contained in $\ODphi$. Actually, if $\varphi(a^2)\varphi(\hlambda) $ converges in the operator norm topology then so does $\varphi(\hlambda)\varphi(ab)$, by $\|(\varphi(\hlambda)- \varphi(\hmu))\varphi(ab)\|^2 $ $= \|\varphi(\hlambda-\hmu)\varphi(ab)\varphi(ba)\varphi(\hlambda -\hmu)\| \leq \|\varphi(\hlambda - \hmu)\varphi(a^2)\varphi(\hlambda-\hmu)\|$ for $\lambda$, $\mu \in\Lambda$ and $b\in \Aplusone$. Then we have that $\limlambda \varphi(\hlambda)\varphi(ab) = \varphi(a)\varphi(b)$ for $b\in \Aplusone$ in the operator norm. 

Since the orthogonality domain $\ODphi$ is a \Cs\ by Proposition \ref{PropODphi} (i),  it can be decomposed into the span of $\ODphi_+^{1}$. By the definition of $\ODphi$, we see that $a\in\ODphi_+^1$ implies $\varphi(a)^2 =\limlambdainfty\varphi(a^2)\varphi(\hlambda)$.
\end{proof}

\section{Examples of $k$-positive order $\varepsilon$ maps}
In the previous section we have seen that the class of order zero maps is explicitly divided into the two cases, positive but not completely positive and completely positive (Corollary \ref{Cor2positiveorderzero}). A well-known example of  positive order zero map, but not 2-positive, is the transposition on a matrix algebra. This section studies the possibility of constructing $k$-positive maps of almost order zero but not $k+1$-positive. 

From now on we denote by $\{\eijn\}_{i,j=1}^n$ the canonical matrix units of $M_n$ and $\trn$ the normalized trace on $M_n$. The following construction of $k$-positive almost order zero maps relies  on Tomiyama's work in \cite{Tomi}.

\begin{example}
Fix a natural number $k$ and $\varepsilon >0$. Let $n$ be a natural number such that $k < n$. For $\lambda\in (0, \infty)$, we let $\psilambda$ be the linear map from $\Mn$ to $\Mn$ defined by 
\[ \psilambda(a) =\lambda\trn(a)\onen +(1-\lambda)a\quad\text{for}\ a\in\Mn.\]
Because of \cite[Theorem2]{Tomi}, we can see that $\psilambda$ is $k$-positive if and only if $\lambda \leq 1+\frac{1}{nk-1}$. We let $\lambda\in (0, \infty)$ be such that $\frac{1}{n(k+1) -1} < \lambda -1 \leq \frac{1}{nk-1}$. 

Let $\iota : \Mn \rightarrow (\eoneonem \otimes\onen)\Mm\otimes\Mn(\eoneonem\otimes\onen)$  be the canonical isomorphism.  We define a linear map $\varphilambdam$ from $\Mm\otimes\Mn$ to $\Mm\otimes\Mn$ by 
\[\varphilambdam (x) = (1-\varepsilon)x + \varepsilon \onem \otimes\psilambda\circ\iota^{-1}((\eoneonem\otimes\onen) x(\eoneonem\otimes\onen)),\quad\text{for}\quad x\in\Mm\otimes\Mn.\] 
Then for any $m\in\N$, this map $\varphilambdam$ is unital and $k$-positive, satisfying 
\[ \llVert \varphilambdam (x)^2 - \varphilambdam(x^2)\rrVert < 6\varepsilon,\]
for any contraction $x$ in $\Mm\otimes\Mn$. By Theorem \ref{FirstMainTheorem} we can regard $\varphilambdam$ as an almost order zero map. 

For a large $m\in\N$, we have that $\varphilambdam$ is not $(k+1)$-positive. Actually, setting the unital completely positive map $\Phin : \Mm\otimes\Mn\rightarrow \Mn$ by $\Phin(a\otimes b)= \trm (a)b$, and $\lambdatilde = \frac{m\varepsilon\lambda}{(1-\varepsilon)+ m\varepsilon} >0$, we see that  
\begin{align*}
\Phin\circ\varphilambdam(\iota(a))&= \frac{1-\varepsilon}{m} a + \varepsilon (\lambda\trn(a)\onen + (1-\lambda)a)\\
&= \frac{\varepsilon\lambda}{\lambdatilde}(\lambdatilde\trn(a)\onen + (1-\lambdatilde)a),\quad\text{for}\quad a\in\Mn.
\end{align*} 
Since $\limminfty \lambdatilde =\lambda\in \left(1+ \frac{1}{n(k+1)-1},\ 1+\frac{1}{nk-1}\right]$, it follows that $\lambdatilde > 1+\frac{1}{n(k+1)-1}$ for a large $m\in\N$. Thus $\Phin\circ\varphilambdam|_{\iota(\Mn)}$ is not $(k+1)$-positive, so $\varphilambdam$ is not. 
\end{example}

In contrast to the above example, by fixing the size of the matrix algebras, the following proposition shows how close unital 2-positive almost order zero maps are to being completely positive.
\begin{proposition}
For $\varepsilon >0$, we let $\delta >0$ be as in Theorem \ref{FirstMainTheorem}. Let $\varphi$ is a unital $2$-positive map from $\Mn$ to a unital \Cs{} B. Suppose that $\|\varphi(a)^2-\varphi(a^2)\|<\delta$ for any positive contraction $a\in\Mn$. Then the linear map $\Mn\ni a\mapsto \varphi(a) + n\varepsilon \Trn(a) \oneB$ is completely positive, 
where $\Trn$ denotes the non-normalized trace on $\Mn$.  
\end{proposition}
\begin{proof}
We set $\bzero = \displaystyle\sum_{i=1}^n \eonein\otimes\eonein\in \Mn\otimes\Mn$ and $b=\bzero^*\bzero\in \Mn\otimes\Mn$. It is enough to show that the Choi matrix $(\varphi+\varepsilon n\Trn)\otimes\idMn(b)$ is a positive element in $B\otimes\Mn$, (see \cite[Proposition 1.5.12]{BO} for example). Since $\llVert\varphi(\eionen)\varphi(\eonejn)-\varphi(\eijn)\rrVert < \varepsilon$, it follows that 
\[\llVert \varphi\otimes\idMn (\bzero)^* \varphi\otimes\idMn(\bzero) - \sum_{i,j=1}^n \varphi(\eijn)\otimes\eijn\rrVert <n\varepsilon .\]
Thus we have that 
\[ (\varphi + n \varepsilon \Trn)\otimes \idMn(b) = \sum_{i,j=1}^{n}\varphi(\eijn)\otimes\eijn + n \varepsilon \sum_{i=1}^n \oneB\otimes\eiin \geq 0.\]
\end{proof}

\section{ One-way CPAP} 
In the rest of this paper, we focus on nuclear \Cs s and aim to show the second main result Theorem \ref{SecondMainTheorem}. The following weaker characterization of nuclearity has implicitly appeared in Ozawa's survey \cite{Oz}, which was obtained in the context of \cite{KS} and \cite{KOS}. Let us revisit this argument for our self-contained proof. 

For a \Cs{} $B$ and a net $\Alambda$, $\lambda\in\Lambda$ of \Css s of $B$, we denote by $\ProdAlambda$ the $\llinfty$-direct sum of $\{\Alambda\}_{\lambda\in\Lambda}$ (i.e., the set of nets $(\alambda)_{\lambda\in\Lambda}$ such that $\alambda\in\Alambda$ and $\suplambda\|\alambda\| <\infty$), and $\OplusAlambda$ the 
$c_0$-direct sum (i.e., the set of $(\alambda)_{\lambda}\in\ProdAlambda$ such that $\limlambdainfty \|\alambda\|=0$). It is well-known that $\ProdAlambda$ is a \Cs{} and $\OplusAlambda$ is an ideal of $\ProdAlambda$. When $\Alambda=A$ for any $\lambda\in\Lambda$ we let 
\[\linftyLambdaA = \ProdAlambda\quad\text{ and}\quad \czeroLambdaA=\OplusAlambda.\]
We identify a \Cs{} $A$ with the \Css{} of $\frac{\linftyLambdaA}{\czeroLambdaA}$ consisting of equivalence classes of constant nets.

\begin{theorem}\label{OneSideCPAP}
A \Cs{} $A$ is nuclear if and only if there exists a net $\varphilambda : \MNlambda \rightarrow A$, $\lambda\in\Lambda$ of completely positive contractions such that the canonical completely positive contraction 
\[\Phi=(\varphilambda)_{\lambda} : \frac{\ProdMNlambda}{\OplusMNlambda} \longrightarrow \frac{\linftyLambdaA}{\czeroLambdaA}\quad\text{ satisfies}\quad \Phi\left(\left(\frac{\ProdMNlambda}{\OplusMNlambda}\right)^1\right) \supset \Aone.\]
\end{theorem}

The following lemma is essentially given in \cite[Lemma 3.5]{KS} for completely positive maps. A generalization for 2-positive maps may be of independent interest.  

For a given unital \Cs{} $A$, we define 
\[\LambdaA=\{ F\subset \Aone\ :\ \text{ a finite subset of unitaries in $A$}\}\times \{\varepsilon\in\R\ :\ \varepsilon >0\},\]
and regard $\LambdaA$ as the (upward-filtering) ordered set by the inclusion order on $2^{\Aone}$ and the reverse order on $\R$. 
For a \Cs{} $A$, we let $\dist (x, F)$ denote $\displaystyle\inf_{y\in F} \|x-y\|$ for $x\in A$ and $F\subset A$. 
\begin{lemma}\label{unitaries}
Let $A$ be a unital \Cs{} and $\cM$ a unital \Cs{} which is closed under the polar decomposition by unitaries, i.e., for any $x\in \cM$ there exists a unitary $u\in\cM$ such that $x=u|x|$. Suppose that for $\lambda=(F, \varepsilon)\in\LambdaA$, a 2-positive contraction $\varphi : \cM\rightarrow A$ satisfies $\dist(x, \varphi(\cMone))<\varepsilon$ for all $x\in F$. Then there exist unitaries $\Ux\in\cM$, $x\in F$ such that 
\[ \llVert\varphi(\Ux)-x\rrVert < 3\sqrt{\varepsilon}\quad\text{for all}\ x\in F.\] 
\end{lemma}
\begin{proof}
Let $\yx\in \cMone$ be such that $\llVert\varphi(\yx)-x\rrVert < \varepsilon$ for $x\in F$. For $x\in F$, by the polar decomposition of $\yx$, there exists a unitary $\Ux\in\cM$ such that $\yx=\Ux|\yx|$. Since $x\in F$ is a unitary, it follows that $\llVert\varphi(\yx)^*\varphi(\yx) -\oneA\rrVert < 2\varepsilon$. Then Kadison's inequality implies that 
\[ (1-2\varepsilon)\oneA\leq\varphi(\yx)^*\varphi(\yx)\leq\varphi(\yx^*\yx)\leq\varphi(\onecM)\leq\oneA.\]
By $\varphi(1-|\yx|)\leq\varphi(1-\yx^*\yx)\leq 2\varepsilon\oneA$, we have that
\begin{align*}
\llVert\varphi(\Ux)-x\rrVert &< \llVert\varphi(\Ux -\yx)\rrVert +\varepsilon=\llVert\varphi(\Ux(1-|\yx|))^*\varphi(\Ux(1-|\yx|))\rrVert^{1/2}+\varepsilon \\
&\leq \llVert \varphi((1-|\yx|)^2)\rrVert^{1/2} +\varepsilon\leq\sqrt{2\varepsilon} + \varepsilon < 3\sqrt{\varepsilon}.
\end{align*}
\end{proof}

\begin{lemma}[Lemma 3.6 of \cite{KS}, see also Lemma 4.1.4 of \cite{Fara}]\label{permutations}
\

\noindent For $N\in\N$ and $(F, \varepsilon)\in \LambdaMN$, there exist unitaries $\vi\in \MN$, $i=1,2,...,K$ and permutations $\sigmax$, $x\in F$ of $\{1, 2, ..., K\}$ such that 
\[ \max_{i=1,2,...,K} \llVert \vi\cdot x- v_{\sigmax(i)}\rrVert < \varepsilon\quad\text{for all } x\in F.\]
\end{lemma}

\

\begin{proof}[Proof of Theorem \ref{OneSideCPAP}]
It is shown in \cite[Theorem]{Kir}, \cite[Theorem 3.1]{CE} that the nuclearity of $A$ implies the completely positive approximation property (CPAP) which is stronger than the condition in Theorem \ref{OneSideCPAP}. Then it is enough to show the converse direction. 

First, the following argument allows us to reduce to the case of unital \Cs\ $A$. Actually
it is well-known that $A$ is nuclear if and only if the unitization $\Atilde$ of $A$ is nuclear. For $\lambdatilde=(\Ftilde, \varepsilon)\in \LambdaAtilde$, taking an approximate unit of $A$ we have a positive contraction $e\in A$ and $\lambdax\in \C$ for $x\in\Ftilde$ such that $(\oneAtilde-e)x\approx_{\varepsilon} \lambdax(\oneAtilde-e)$ and $[x, e]\approx_{\varepsilon} 0$ for all $x\in \Ftilde$. Let $\etilde\in \Aplusone$ be such that $e^{1/2}\etilde^{1/2} \approx_{\varepsilon} e^{1/2}$. By the assumption of $A$, we now obtain a completely positive contraction $\varphi : \MN \rightarrow A$ such that $\dist (y, \varphi(\MN^1)) <\varepsilon$ for all $y\in\{\etilde\}\cup\{ \etilde^{1/2}x\etilde^{1/2}\ :\ x\in \Ftilde\} \subset \Aone$. Then we have $e^{1/2}\varphi(\oneMN)e^{1/2}\approx_{3\varepsilon} e$. Define a completely positive map $\varphitilde : \MN\oplus \C\rightarrow A$ by $\varphitilde(x\oplus c) = e^{1/2}\varphi(x)e^{1/2}+ c(\oneAtilde -e)$ for $x\in \MN$ and $c\in \C$. Since $\varphitilde (\oneMN \oplus 1) \approx_{3\varepsilon} \oneAtilde$, the canonical extension $\varphilambdatilde : M_{N+1} \rightarrow A$ of $\frac{1}{1+3\varepsilon}\varphitilde$ is a completely positive contraction, which satisfies the condition in Theorem \ref{OneSideCPAP} for $\Atilde$.

Let $\lambda=(F, \varepsilon)\in \LambdaA$ be such that $\varepsilon <1 $. By the assumption, we now obtain a completely positive contraction $\varphi : \MN \rightarrow A$ such that $\dist (x, \varphi(\MN^1))< \left(\varepsilon/6\right)^4$ for all $x\in F$. By Lemma \ref{unitaries}, there are unitaries $\Ux \in \MN$, $x\in F$ such that $\llVert\varphi(\Ux)-x\rrVert < \varepsilon^2/12$ for $x\in F$. By Lemma \ref{permutations}, for $\left(\{\Ux\}_{x\in F},\ \varepsilon /2 \right)\in \LambdaMN$, there exist unitaries $\vi\in\MN$, $i=1,2,...,K$ and permutations $\sigmax$, $x\in F$ of $\{ 1,2,..., K\}$ such that 
\[ \llVert \vi\cdot \Ux -v_{\sigmax(i)}\rrVert < \varepsilon /2\quad\text{for all }i=1,2,..., K,\text{ and } x\in F.\]
Due to the Kasparov-Stinespring dilation theorem \cite{Kas}, (see also \cite[Theorem 6.5]{Lan}), there exists a $*$-homomorphism $\pi : \MN \rightarrow \BHA$ such that $\varphi(a) =\pi(a)_{1,1}\in A$, where the notations of $\HA$ and $\aij\in A$ for $a\in \BHA$ are same as in Lemma \ref{PositiveUnitary}. We set $\aji = \pi(\vi)_{j,1}\in A$ for $i=1,2,..., K$ and $j\in\N$. 

From {\rm (ii)} of Lemma \ref{PositiveUnitary} and $\llVert\pi(\Ux)_{1,1}^*\pi(\Ux)_{1,1} -\oneA\rrVert=\llVert\varphi(\Ux)^*\varphi(\Ux)-\oneA\rrVert < \varepsilon^2/6$ it follows that 
\[\llVert\sum_{j=2}^\infty\pi(\Ux)_{j,1}^*\pi(\Ux)_{j,1}\rrVert < \varepsilon^2/6\quad\text{for all }x\in F.\]
Combining this with $\llVert\pi(\vi)\cdot\pi(\Ux)-\pi(v_{\sigmax(i)})\rrVert< \varepsilon /2$, we have that for $x\in F$ 
\[ \llVert\sum_{j=1}^{\infty} \left\vert\aji x - a_{j}^{(\sigmax(i))}\right\vert^2\rrVert^{1/2} = \llVert (\aji x)_j - (a_{j}^{(\sigmax(i))})_j\rrVert_{\HA}< \varepsilon.\]
Since $\vi$, $i=1,2,...,K$ are unitaries, we obtain $L\in\N$ such that 
\[ \llVert \sum_{j=1}^L \aji^*\aji -\oneA\rrVert < \varepsilon.\]

Let $\Asecond$ be the second dual of $A$ faithfully represented on a Hilbert space $\cH$  i.e., $A\subset \Asecond \subset \BcH$. For $\lambda\in\LambdaA$, we define a completely  positive map $\Philambda : \BcH\rightarrow\BcH$ by 
\[ \Philambda(y) =\frac{1}{K}\sum_{i=1}^K\sum_{j=1}^L \aji^* y\ \aji\quad\text{for } y\in\BcH.\]
Thus we have that for $x\in F$ and $y\in \BcH^1$ 
\begin{align*}
\Philambda(y)x&\approx_{\varepsilon} \frac{1}{K}\sum_{i=1}^K\sum_{j=1}^L \aji^*y\ a_j^{(\sigmax(i))} \\ 
&=\frac{1}{K}\sum_{i=1}^K\sum_{j=1}^L {a_j^{(\sigmax^{-1}(i))}}^*y\ \aji\\
&\approx_{\varepsilon} \frac{1}{K}\sum_{i=1}^K\sum_{j=1}^L (\aji x^*)^*y\ \aji = x\Philambda(y).
\end{align*}
From $\displaystyle \llVert\sum_{j=1}^L\aji^* \aji -\oneA\rrVert < \varepsilon$, for $y\in \BcH^1\cap A'$ it follows that $\Philambda(y)\approx_{\varepsilon} y$. So, $\Philambda$ is close to a conditional expectation onto $A'$. Let $\omega$ be a (cofinal) ultrafilter on the ordered set $\LambdaA$. Then one can define a bounded map $\Phi : \BcH\rightarrow \BcH$ by the weak$^*$ limit $\Phi(y) = \text{weak$^*$-}\limlambdaomega\Philambda(y) $ in $\BcH$. By the above conditions of $\Philambda$, it is straightforward to check that $\Phi$ is a conditional expectation on $\BcH \cap A'$. Hence $A'$ is an injective von Neumann algebra, and so is $A''=\Asecond$. Because of \cite{Con}, we can see that $\Asecond$ is AFD which implies the CPAP of $A$. 
\end{proof}

\begin{remark}
In \cite{Smi} R. Smith showed that the complete positivity of contractive maps in the CPAP can be replaced by the complete contractivity. However, we cannot expect to replace completely positive contractions $\varphilambda$ in Theorem \ref{OneSideCPAP} by completely contractive maps. In fact, there are many non-nuclear \Cs s with the completely contractive approximation property (CCAP), although  any \Cs{} $A$ with the CCAP satisfies the following condition : there exists a net of complete contractions $\varphilambda : \MNlambda\rightarrow A$, $\lambda\in\Lambda$ such that for $a\in \Aone$ there are $\xalambda\in \MNlambda^1$, $\lambda\in\Lambda$ satisfying $\limlambdainfty \varphilambda(\xalambda)=a$. 
\end{remark}

\section{Decomposition rank by 2-positive maps}
Before proving Theorem \ref{SecondMainTheorem}, let us recall the definition of decomposition rank. 
\begin{definition}[E. Kirchberg - W. Winter, \cite{KW}]\label{Decompositionrank}
For $d\in \N\cup\{0\}$, a \Cs{} $A$ is said to have {\it decomposition rank at most $d$}, if for a finite subset $F$ of contractions in $A$ and $\varepsilon >0$, there exist finite dimensional \Cs s $\Fi$, $i=0,1,...,d$, a completely positive contraction $\psi : A\rightarrow \bigoplus_{i=0}^d \Fi$, and completely positive order zero contractions $\varphii : \Fi \rightarrow A$, $i=0,1,...,d$ such that $\sum_{i=0}^d \varphii : \bigoplus_{i=0}^d \Fi\rightarrow A$ is contractive and 
\[ \llVert\left(\sumvarphii\right)\circ\psi(x) -x\rrVert < \varepsilon,\quad\text{for all }x\in F.\]
\end{definition}

\begin{theorem}[ Theorem \ref{SecondMainTheorem} ]\label{SecondMainTheorem2}
Let $A$ be a unital separable \Cs{} and $d\in\N\cup\{0\}$. Then the following conditions are equivalent. 
\begin{enumerate}
\item The decomposition rank of $A$ is at most $d$. 
\item For $\lambda=(F, \varepsilon)\in \LambdaA$, there are finite dimensional \Cs s $\Fi$, $i=0,1,...,d$, a 2-positive contraction $\psi : A\rightarrow \bigoplus_{i=0}^d\Fi$, and 2-positive order zero contractions $\varphii : \Fi \rightarrow A$, $i=0,1,...,d $ such that $\sum_{i=0}^d\varphii :  \bigoplus_{i=0}^d\Fi\rightarrow A$ is contractive and 
\[ \llVert\left(\sumvarphii\right) \circ \psi(x) -x\rrVert < \varepsilon,\quad\text{for all }x\in F.\]
\item There exist finite dimensional \Cs s $\Filambda$, $i=0,1,...d$, $\lambda\in \Lambda$ and nets $\varphiilambda : \Filambda \rightarrow A$,  $i=0,1,...,d$, $\lambda\in \Lambda$ of 2-positive order zero contractions such that $\sumvarphiilambda : \Flambda\rightarrow A$ is contractive for any $\lambda\in\Lambda$, where $\Flambda =\OplusFilambda$, and the canonical contraction 
\[\Phi =\left(\sumvarphiilambda\right)_{\lambda} :\  \frac{\ProdFlambda}{\OplusFlambda}\longrightarrow\frac{\linftyLambdaA}{\czeroLambdaA}\quad \text{satisfies}\quad \Phi\left(\left(\frac{\ProdFlambda}{\OplusFlambda}\right)^1\right) \supset \Aone. \]
\end{enumerate}
\end{theorem}
\begin{proof}
The implications {\rm (i)} $\Longrightarrow$ {\rm (ii)} $\Longrightarrow$ {\rm (iii)} are trivial. We shall show {\rm (iii)} $\Longrightarrow$ {\rm (i)}. By Corollary \ref{Cor2positiveorderzero}, we see that $\sumvarphiilambda$, $\lambda\in\Lambda$ are completely positive contractions. Taking a conditional expectation from a matrix algebra onto $\Flambda$, by Theorem \ref{OneSideCPAP} we know that $A$ is nuclear. 

From the assumption of {\rm (iii)}, for $\mu=(F, \varepsilon)\in \LambdaA$ we obtain finite dimensional \Cs s $\Fimu$, $i=0,1,...,d$, and completely positive order zero contractions $\varphiimu : \Fimu \rightarrow A$, $i=0,1,...,d$ such that 
\[ \dist\left( x, \ \sumvarphiimu\left(\left(\OplusFimu\right)^1\right)\right)<\varepsilon, \quad\text{for all }x\in F.\]
Set $\Fmu =\OplusFimu$ and $\varphimu =\sumvarphiimu : \Fmu \rightarrow A$ for $\mu\in \LambdaA$. By Lemma \ref{unitaries} and $\llVert\varphimu \rrVert \leq 1$, there are unitaries $\Uxmu\in \Fmu$, $x\in F$, $\mu =(F, \varepsilon)\in \LambdaA$,  such that $\llVert\varphimu(\Uxmu) -x\rrVert < 3\sqrt{\varepsilon}$ for all $x\in F$.  For any unitary $x\in A$, we set $\Uxmu =1_{\Fmu}$ if $x\not\in F$ and $\mu =(F, \varepsilon)$, and set $\Ux = \left(\Uxmu\right)_{\mu}\in \ProdFmu$.
We let $Q : \ProdFmu \rightarrow\frac{\ProdFmu}{\OplusFmu}$ be the quotient map, $\barUx=Q(\Ux)$, and let $\cC$ be the \Css{} of $\frac{\ProdFmu}{\OplusFmu}$ generated by $\left\{\barUx\ :\ x \text{ is a unitary in }A\right\}$. 

Let $\varphi : \frac{\ProdFmu}{\OplusFmu}\rightarrow \frac{\linftyLambdaAA}{\czeroLambdaAA}$ be the completely positive contraction defined by $\varphi\circ Q((\xmu)_{\mu}) =(\varphimu(\xmu))_{\mu}$ in $\linftyLambdaAA/ \czeroLambdaAA$.  By regarding $A$ as the \Css{} of $\linftyLambdaAA /\czeroLambdaAA$, it follows that $\varphi\left(\barUx\right)=x$ for any unitary $x\in A$, then $\varphi\left(\cC\right)=A$. Because of  
\[ \varphi\left(\barUx\right)^*\varphi\left(\barUx\right) =\oneA =\varphi\left(\barUx^*\ \barUx\right)\quad\text{and}\quad \varphi\left(\barUx\right)\varphi\left(\barUx\right)^*=\oneA=\varphi\left(\barUx\ \barUx^*\right),\]
 we see that $\varphi|_{\cC} : \cC\rightarrow A$ is a unital $*$-homomorphism. 
Let $\varphitilde$ be the $*$-isomorphism from $\cC /\ker (\varphi |_{\cC})$ onto $A$ and $\psitilde =\varphitilde^{-1}$. 

Applying the Choi-Effros lifting theorem \cite{CE2} to $\psitilde$, we obtain a unital completely positive map $\psi : A\rightarrow \ProdFmu$ such that $\varphi\circ Q\circ\psi(a) =a$ for any $a\in A$. Note that $A$ is required to be nuclear and separable in order to apply \cite[Theorem 3.10]{CE2}. Taking  unital completely positive maps $\psimu : A\rightarrow \Fmu$, $\mu\in \LambdaA$ with $(\psimu(a))_{\mu} =\psi(a)$ for $a\in A$, we conclude that $\psimu$ and $\varphiimu$ satisfy the conditions in {\rm (i)}.
\end{proof}

\noindent{\bf Acknowledgements.}\quad  The author would like to thank Professor Marius Dadarlat for helpful comments on this research, and Professor Narutaka Ozawa for showing him the valuable survey \cite{Oz}. He also expresses his gratitude to Professor Huaxin Lin and the organizers of Special Week on Operator Algebras 2019 for their kind hospitality during the author's stay in East China Normal University. This work was supported in part by the Grant-in-Aid for Young Scientists (B) 15K17553, JSPS.

\noindent Yasuhiko Sato \\
Graduate School of Science \\
Kyoto University \\
Sakyo-ku, Kyoto 606-8502\\ 
Japan \\
ysato@math.kyoto-u.ac.jp

\end{document}